\documentclass[11pt]{article}
\usepackage{}

\usepackage[total={6in, 8.5in}]{geometry}
\usepackage{amsfonts}
\usepackage{amsthm}
\usepackage{amssymb}
\usepackage{amsmath}
\usepackage{enumerate}
\usepackage[
pdfauthor={ESYZ},
pdftitle={Toughness and spanning trees in K4mf graphs},
pdfstartview=XYZ,
bookmarks=true,
colorlinks=true,
linkcolor=blue,
urlcolor=blue,
citecolor=blue,
bookmarks=false,
linktocpage=true,
hyperindex=true
]{hyperref}

\usepackage{tkz-graph}
\usetikzlibrary{arrows}
\usepackage{tkz-euclide}
\usetkzobj{all}

\usepackage{xcolor}
\long\def\ignore#1{}

\newtheorem{THM}{Theorem}[section]

\newtheorem{LEM}[THM]{Lemma}
\newtheorem{CON}[THM]{Conjecture}

\newtheorem{COR}[THM]{Corollary}

\newtheorem{CLA}{Claim}[section]
\newcommand{\pf}{\textbf{Proof}.\quad}

\newcommand{\qqed}{\hfill $\blacksquare$\vspace{1mm}}

\newcommand{\ve}{\varepsilon }

\newcommand{\CC}{\mathcal{C}}

\newcommand{\pbar}{\overline{\varphi}}
\newcommand{\pb}{\varphi^{\text{bad}}}
\newcommand{\pbo}{\varphi_1^{\text{bad}}}

\begin{document}
\title{Independence number of edge-chromatic critical graphs}

\author{%
	Yan Cao$^{\dag}$ 
	\qquad Guantao Chen$^{\dag}$ \qquad Guangming Jing$^{\dag}$ \qquad Songling Shan$^{\ddag}$\\
	$^{\dag}$Department of Mathemetics and Statistics, 25 Park Place, 14th Floor\\
	Georgia State Univeristy,
	Atlanta, GA 30303\\
		$^{\ddag}$Department of Mathematics, 1326 Stevenson Center\\
	Vanderbilt University, Nashville, TN 37240
}

\date{\today}
\maketitle

 \begin{abstract}
Let $G$ be a simple graph with maximum degree $\Delta(G)$ and chromatic index $\chi'(G)$. A classic result of Vizing indicates that either $\chi'(G )=\Delta(G)$ or $\chi'(G )=\Delta(G)+1$. The graph $G$ is called $\Delta$-critical if $G$ is connected, 
$\chi'(G )=\Delta(G)+1$ and for any $e\in E(G)$,  $\chi'(G-e)=\Delta(G)$.  Let $G$ be an $n$-vertex $\Delta$-critical graph. Vizing conjectured that $\alpha(G)$, the independence number of $G$, is at most $\frac{n}{2}$. The current best result on this conjecture, shown by Woodall, is that $\alpha(G)<\frac{3n}{5}$. We show  that for any given $\ve\in (0,1)$, there exist positive constants $d_0(\ve)$ and $D_0(\ve)$ such that if $G$ is an $n$-vertex 
 $\Delta$-critical graph with minimum degree at least $d_0$ and maximum degree at least $D_0$,
 then $\alpha(G)<(\frac{{1}}{2}+\ve)n$. 
In particular, we show that if $G$ is an $n$-vertex $\Delta$-critical graph with  minimum degree at least $d$ and $\Delta(G)\ge (d+2)^{5d+10}$, then 
\[
\alpha(G) < \left.
\begin{cases}
\frac{7n}{12}, & \text{if $d= 3$; }  \\
\frac{4n}{7}, & \text{if $d= 4$; } \\
\frac{d+2+\sqrt[3]{(d-1)d}}{2d+4+\sqrt[3]{(d-1)d}}n<\frac{4n}{7}, & \text{if $d\ge 19$. }
\end{cases}
\right.
\]


 \end{abstract}

\vspace{2mm}

\section{Introduction}

All graphs considered are simple and finite. Let $G$ be a graph.
We denote by $V(G)$ and $E(G)$ the vertex set and edge set of $G$, respectively. 
For a vertex $v\in V(G)$, $N_G(v)$ is the set of neighbors of $v$ in $G$, and 
 $d_G(v)=|N_G(v)|$ is the degree of vertex $v$ in $G$. We simply write $N(v)$ and 
 $d(v)$ if $G$ is clear. 
For $e\in E(G)$, $G-e$
denotes the graph obtained from $G$ by deleting the edge $e$. 
We reserve the symbol $\Delta$ for $\Delta(G)$, the maximum degree of $G$
throughout  this paper.  The {\it independence number\/} of $G$, denoted $\alpha(G)$, 
is the largest size of an independent set in $G$.

An {\it edge $k$-coloring\/} of $G$ is a mapping $\varphi$ from $E(G)$ to the set of integers
$[1,k]:=\{1,\cdots, k\}$, called {\it colors\/}, such that  no adjacent edges receive the same color under $\varphi$. The {\it chromatic index\/} of $G$, denoted $\chi'(G)$, is defined to be the smallest integer $k$ so that $G$ has an edge $k$-coloring. 
We denote by $\CC^k(G)$ the set of all edge $k$-colorings of $G$. 
In 1965, Vizing~\cite{Vizing-2-classes} showed that a graph of maximum degree
$\Delta$ has  chromatic index either $\Delta$ or $\Delta+1$.
If $\chi'(G)=\Delta$, then $G$ is said to be of {\it class 1\/}; otherwise, it is said to be
of {\it class 2\/}.  
Holyer~\cite{Holyer} showed that it is NP-complete to determine whether an arbitrary graph is of class 1. Similar to vertex coloring, it is essential to color the ``core'' part of a graph and then extend the coloring to the whole graph without increasing the total number of colorings. This leads to the concept of {\it edge-chromatic criticality\/}. A graph $G$ is called {\it edge-chromatic critical\/}
if for any proper subgraph $H\subseteq G$, $\chi'(H)<\chi'(G)$. 
We say $G$ is {edge-$\Delta$-critical} or simply {\it $\Delta$-critical\/} if $G$
is edge-chromatic critical and $\chi'(G)=\Delta+1$. It is clear that $G$ is $\Delta$-critical if and only if $G$ is connected with $\chi(G)=\Delta+1$ and $\chi'(G-e)=\Delta$, for any edge $e\in E(G)$.   By this definition,  every class 2 graph with maximum degree $\Delta$ can be reduced to a $\Delta$-critical graph by removing edges or vertices. Vizing conjectured that $\Delta$-critical graphs have some special structural properties. In particular, he proposed the following conjectures.

\begin{CON}[Vizing's Independence Number Conjecture~\cite{vizing-ind}]\label{ind}
	Let $G$ be a $\Delta$-critical graph of order $n$.
	Then $\alpha(G)\le n/2$.
\end{CON}

\begin{CON}[Vizing's 2-factor Conjecture~\cite{vizing-2factor}]\label{2-factor}
	Let $G$ be a $\Delta$-critical graph. Then $G$ contains a 2-factor;
	that is, a 2-regular subgraph $H$ of $G$ with $V(H)=V(G)$.
\end{CON}

\begin{CON}[Vizing's Average Degree Conjecture~\cite{vizing-2factor}]\label{ave-deg-con}
	Let $G$ be an $n$-vertex $\Delta$-critical graph. Then the average degree of $G$ is at least $\Delta-1+\frac{n}{3}$. 
\end{CON} 

The above conjectures do not hold for edge-chromatic critical class 1 graphs such as 
a length 2 path. Also, they do not hold for class 2 graphs which are not $\Delta$-critical. Partial results have been obtained for each of these conjectures. In this paper, we investigate 
Vizing's Independence Number Conjecture. This conjecture was confirmed for special graph classes including graphs with many edges such as overfull graphs by Gr{\"u}newald and Steffen~\cite{vizing-2-factor-overful}, and  $n$-vertex $\Delta$-critical graphs $G$ with $\Delta\ge \frac{n}{2}$ by Luo and Zhao~\cite{ind2}.  Let $G$ be an $n$-vertex
$\Delta$-critical graph.
Brinkmann et al.~\cite{ind1}, in 2000,
proved that  $\alpha(G)<2n/3$; and
the upper bound is further improved when
the maximum degree is between 3 to 10. Luo and Zhao~\cite{ind2}, in 2008, by improving the result of
Brinkmann et al., showed that 
$\alpha(G)<(5\Delta-6)n/(8\Delta-6)<5n/8$  when $\Delta\ge6$.
In 2009,
Woodall~\cite{ind3} further improved the upper bound of $\alpha(G)$ to
$3n/5$. In this paper, by using new adjacency lemmas, we obtain the following results. 

\begin{THM}\label{THM-main0}
For any given $\ve\in (0,1)$, there exist positive constants $d_0(\ve)$ and $D_0(\ve)$ such that if $G$ is an $n$-vertex 
$\Delta$-critical graph with minimum degree at least $d_0$ and maximum degree at least $D_0$,
then $\alpha(G)<(\frac{{1}}{2}+\ve)n$. 
\end{THM}

Theorem~\ref{THM-main0} is implied by the following result. 
\begin{THM}\label{THM-main}
	Let 
	 $G$ be an $n$-vertex $\Delta$-critical graph with minimum degree at least $d$
	 and $\Delta\ge (d+2)^{5d+10}$. Then 
	\[
	\alpha(G) < \left.
	\begin{cases}
	\frac{7n}{12}, & \text{if $d=3$; }  \\
	\frac{4n}{7}, & \text{if $d=4$; } \\
	\frac{d+2+\sqrt[3]{(d-1)d}}{2d+4+\sqrt[3]{(d-1)d}}n, & \text{if $ d\ge 19$. }
	\end{cases}
	\right.
	\]
\end{THM}
When $d\ge 19$, $\frac{d+2+\sqrt[3]{(d-1)d}}{2d+4+\sqrt[3]{(d-1)d}}<\frac{4}{7}$.
In fact, we suspect the following might be true.

\begin{CON}
		Let  $d\ge 2$ be a positive integer. Then there exists a constant $D_0$ depending only on $d$ such that if
		$G$ is an $n$-vertex $\Delta$-critical graph  with $\delta(G)\ge d$ and $\Delta \ge D_0$,  then 
		$$
		\alpha(G)<\frac{d+4}{2d+6}n.
		$$
\end{CON}
The case for $d=2$ was confirmed by   Woodall's result~\cite{ind3}, and the cases for $d=3,4$ are covered in Theorem~\ref{THM-main}. 

The remainder of the paper is organized as follows. We introduce some edge-coloring notation and technique lemmas in Section 2, and we prove Theorem~\ref{THM-main} in Section 3. 

\section{Technique Lemmas}

In this  section, we list the classic Vizing's Adjacency Lemma and some new developed 
adjacency lemmas which will be used for proving Theorem~\ref{THM-main}. 

\begin{LEM}[Vizing's Adjacency Lemma]\label{vizing adjacency lemma}
	If $G$ is a $\Delta$-critical graph, then for any edge
	$xy\in E(G)$, $y$ is adjacent to at least $\Delta-d(x)+1$ vertices $z$ of degree  $\Delta$
	 with $z\ne x$.
\end{LEM}


Let $G$ be a $\Delta$-critical graph, $x\in V(G)$, $y,z\in N(x)$, and  $\varphi \in \CC^\Delta (G-xy)$. 
For any $v\in V(G)$, the set of colors present at $v$ is 
$\varphi(v)=\{\varphi(e)\,:\, \text{$e$ is incident to $v$}\}$, and the set of colors missing at $v$ is $\pbar(v)=[1,\Delta]-\varphi(v)$, where $[1,\Delta]=\{1,2,\cdots, \Delta\}$.  
Note that by the $\Delta$-criticality of $G$, $\pbar(x)\cap \pbar(y)=\emptyset$. Let $q$ be a positive integer.  
Define
\begin{eqnarray*}
\sigma_q(x,z)&=&|\{u\in N(z)-\{x\}\,:\, d(u)\ge q\}|,\\
A_\varphi(x,y,q)&=&\{\alpha\in \varphi(x)\,:\, \text{$\exists$ $u\in N(y)$ such that $\varphi(yu)=\alpha$ and $d(u)<q$}\},\\
B_\varphi(x,y,q)&=&\{\beta\in \varphi(x)\,:\, \text{$\exists$ $u\in N(y)$ such that $\varphi(yu)=\beta$ and $d(u)\ge q$}\},\\
M_\varphi(x,y,q)&=&A_\varphi(x,y,q)\cup \pbar(x)\cup \pbar(y),\\
N(x, M_\varphi(x,y,q))&=&\{z\in N(x)\,:\, \varphi(xz)\in M_\varphi(x,y,q)\},\\
N(x, B_\varphi(x,y,q))&=&\{z\in N(x)\,:\, \varphi(xz)\in B_\varphi(x,y,q)\}.
\end{eqnarray*}
We simply write $A_\varphi(q), B_\varphi(q)$ and $M_\varphi(q)$ if $xy$ is specified and clear. By the definitions, we have that 
\begin{eqnarray*}
A_\varphi(x,y,q)\cup B_\varphi(x,y,q)&=& \varphi(x)\cap \varphi(y),\\
N(x, M_\varphi)\cup N(x, B_\varphi) \cup \{y\}&=& N(x).
\end{eqnarray*}	
For $v\in V(G)$, let 
$$
\pb(v)=\{\alpha \in [1, \Delta]\,:\, \text{either $\alpha\in \pbar(v)$ or $\exists$  $v'\in N(v)$ so that $\varphi(vv')=\alpha$ and $d(v')<q$}\}.
$$
For any $\ve, \lambda\in (0,1)$, define 
\begin{eqnarray}
c_0&=&\lceil \frac{1-\ve}{\ve}\rceil \label{c0},\\
f(\ve)&=&\frac{3c_0^5+19c_0^4+34c_0^3+27c_0^2+11c_0+2}{\ve}, \label{fepsilon}\\
N&=&(c_0+2)(\frac{1}{\lambda}+1)^{3c_0+2} \label{N}, \\
D_0&=&\max\{f(\ve), \frac{3c_0+2}{\lambda^2}, \frac{N+1}{\ve^3}\}. \label{D0}
\end{eqnarray}
%
\begin{LEM}[Corollary 10, \cite{avgdegree}]\label{LEM:largedegree}
	Let $\ve,\lambda\in (0,1)$,  $\Delta$ be a real number with $\Delta\ge D_0$, and 
	let $q=(1-\ve)\Delta$. 
If $G$ is a $\Delta$-critical graph, $xy\in E(G)$ with $d(x)<q$, and $\varphi\in \CC^\Delta(G-xy)$,  then except at most $N$
vertices in $N(x, M_\varphi)$, for each other remaining vertex $z\in N(x, M_\varphi)$,
$$
|\pb(z)-\{\varphi(xz)\}|<\lambda \Delta.
$$

\end{LEM}

\begin{LEM}[Corollary 15, \cite{avgdegree}]\label{LEM:smalldegree}
	Let $G$ be a $\Delta$-critical graph, $\ve\in (0,1)$  and $q=(1-\ve)\Delta$, and let
$xy\in E(G)$ with $d(x)<\ve\Delta$ and $\varphi\in \CC^\Delta(G-xy)$.
	Then for any $z\in N(x,M_\varphi)$,
	$$
	\pb(z)-\{\varphi(xz)\}\subseteq B_\varphi(q).
	$$
	Moreover, for any distinct $z_1,z_2\in  N(x,M_\varphi)$, 
	$$
	\pb(z_1)\cap \pb(z_2)=\emptyset. 
	$$
\end{LEM}

\begin{LEM}[Lemma 13, \cite{avgdegree}]\label{LEM:degree3}
				Let $G$ be a $\Delta$-critical graph, $\ve\in (0,1)$  and $q=(1-\ve)\Delta$, and let
			$xy\in E(G)$ with $d(x)<\ve\Delta$ and $\varphi\in \CC^\Delta(G-xy)$. For $z\in N(x, M_\varphi)$ and $\beta \in B_\varphi(x,y,q)$, if $\beta \in \pb(z)$, then for any $\alpha \in M_\varphi(x,y,q)$, there exist $z'\in N(x, B_\varphi)$ and $u\in N(z')$ such that $\varphi(xz')=\beta$, $\varphi(z'u)=\alpha$
	and $d(u)\ge q$. 
\end{LEM}	

As a special case of Lemma~\ref{LEM:degree3}, we get the following result. 
\begin{COR}\label{p=1}
		Let $G$ be a $\Delta$-critical graph, $\ve\in (0,1)$  and $q=(1-\ve)\Delta$, and let
	$xy\in E(G)$ with $d(x)<\ve\Delta$ and $\varphi\in \CC^\Delta(G-xy)$.  If $|B_\varphi(q)|=1$ and there exists 
	$z\in N(x, M_\varphi)$ such that $\pb(z)\cap B_\varphi(q) \ne \emptyset$, then for any 
	$z'\in N(x, B_\varphi)$  and any $u\in N(z')-\{x\}$, it holds that $d(u)\ge q$.  
\end{COR}


\begin{LEM}\label{bbcolor}
	Let $G$ be a $\Delta$-critical graph, $\ve\in (0,1)$  and $q=(1-\ve)\Delta$, and let
	$xy\in E(G)$ with $d(x)<\ve\Delta$ and $\varphi\in \CC^\Delta(G-xy)$.
	 Suppose $z\in N(x,M_\varphi)$, $\beta \in \pb(z)\cap B_\varphi(x,y,q)$,
	 and  $w\in N(x,B_\varphi)$ such that $\varphi(xw)=\beta$. Define $B_w(\varphi)=\{\varphi(ww')\,:\, w'\in N(w), \varphi(ww')\in B_\varphi(x,y,q)-\{\beta\} \,\, \text{and $d(w')<q$}\}$.
	   Then for any $\beta'\in B_w(\varphi)$, there exists 
	   $z'\in N(x,B_\varphi)$ 
	   and $u\in N(z')$ 
	   such that $\varphi(xz')=\beta'$,  $\varphi(z'u)=\beta$, and  $d(u)\ge q$. 
\end{LEM}
\begin{proof}
	A coloring $\varphi'\in \CC^\Delta(G-xy)$ is called {\bf valid} if 
	$M_{\varphi'}(x,y,q)=M_{\varphi}(x,y,q), B_{\varphi'}(x,y,q)=B_{\varphi}(x,y,q)$,
	and $B_w(\varphi')=B_w(\varphi)$. 
	Let $z \in N(x,M_\varphi)$ and let $\beta \in \pb(z)\cap B_\varphi(x,y,q)$. 
	Assume that $\varphi(xz)=\alpha$. We may {\bf assume that ${\mathbf \alpha \in \pbar(y)}$}. Otherwise,  $\alpha \in A_\varphi(x,y,q)$.
	Let $v\in N(y)$ such that $\varphi(yv)=\alpha$ and $\gamma \in \pbar(y)$.
	If $\gamma\in \pbar(v)$, we recolor the edge $yv$ using  the color $\gamma$ and let $\varphi_1$ be the new coloring of $G-xy$. It is clear that 
	for any edge $e\in E(G-xy)$ with $e\ne yv$, $\varphi(e)=\varphi_1(e)$. 
	Furthermore,  $\varphi_1$ is a valid coloring. However, under $\varphi_1$, $\alpha \in \pbar(y)$. 	So we assume that $\gamma\in \varphi(v)$. 
	Since $d(x)<\ve \Delta$ and $d(v)<q=(1-\ve)\Delta$, there exits a color 
	$\delta \in \pbar(x)\cap \pbar(v)$.  Note that $\gamma \in \varphi(x)$ 
	and $\delta \in \varphi(y)$ by the $\Delta$-criticality of $G$. 
	Let $P_v(\gamma, \delta), P_x(\gamma ,\delta)$, and $ P_y(\gamma ,\delta)$ be the paths 
    	induced by the two colors $\gamma $ and $\delta$ which starts at $v,x$ and $y$, respectively. 
        We claim that $P_x(\gamma ,\delta)= P_y(\gamma, \delta)$. 
       For otherwise, let $\varphi_1$ be the new coloring of $G-xy$ obtained by switching the colors  $\gamma $ and $\delta$ on the path $P_x(\gamma, \delta)$. Then $\varphi_1$ is a $\Delta$-coloring of $G-xy$ such that 
       $\gamma \in \pbar_1(x)\cap \pbar_1(y)$. Now coloring  the edge $xy$ using the color $\gamma$ gives a $\Delta$-coloring of $G$, showing a contradiction to the assumption that $\chi'(G)=\Delta +1$. 
       Thus, $P_x(\gamma ,\delta)= P_y(\gamma, \delta)$. This implies that 
       $P_v(\gamma, \delta)$ is vertex-disjoint from $P_x(\gamma ,\delta)$. 
       We let $\varphi_1$ be the new coloring of $G-xy$ obtained by switching the colors  $\gamma $ and $\delta$ on the path $P_v(\gamma, \delta)$.
       We now have that $\gamma \in \pbar_1(v)$. Since the switching of colors on 
       $P_v(\gamma, \delta)$ does not affect the colors on the edges incident to $y$, we still have that $\gamma \in \pbar_1(y)$. 
       Let $\varphi_2$ be the new coloring of $G-xy$ obtained from $\varphi_1$ by recoloring the edge $yv$ using the color $\gamma$, we see that $\alpha \in \varphi_2(y)$. Because $\delta, \gamma, \alpha \in M_\varphi(x,y,q)$ and to get $\varphi_2$,  we only switched the two colors $\gamma $ and $\delta$ on the path 
       $P_v(\gamma, \delta)$, and then changed the color on the edge $yv$ from $\alpha$ to   $\gamma$, 
       $\varphi_2$ is a valid coloring. Furthermore, $x\not\in V(P_v(\gamma, \delta))$, for any edge $e$ which is incident to 
       $x$ or the vertex $u\in N(x)$ such that $\varphi(xu)=\gamma$, $\varphi_2(e)=\varphi(e)$. Thus, we can use $\varphi_2$ as a coloring 
       for $G-xy$ which satisfies that ${\mathbf \alpha \in \pbar_2(y)}$.
       
       We now take $z \in N(x,M_\varphi)$,    $\beta \in \pb(z)\cap B_\varphi(x,y,q)$ such that $\varphi(xz)=\alpha$ and   $\alpha \in \pbar(y)$.  We take the color $\alpha$ on the edge $xz$ 
       out and color the edge $xy$ using the color $\alpha$, and we get a coloring $\varphi_1$ of $G-xz$.  We see that $\beta \in A_{\varphi_1}(x,z,q)$ and $\alpha \in \pbar_1(z)$. 
       Since $\varphi_1 (e)=\varphi(e)$ for any $e\not\in \{xy, xz\}$, for the specified vertex  $w\in N(x,B_\varphi(x,y,q))$ such that $\varphi(xw)=\beta$, we still have that $\varphi_1(xw)=\beta$, and $B_w(\varphi)=B_w(\varphi_1)$. 
       Since $\beta \in A_{\varphi_1}(x,z,q)$, by Lemma~\ref{LEM:smalldegree}, $B_w(\varphi)\subseteq \pbo(w)$ and $B_w(\varphi)\subseteq B_{\varphi_1}(x,z,q)$. By Lemma~\ref{LEM:degree3}, we know that 
     for any $\beta'\in B_w(\varphi_1)$, there exists 
     $z'\in N(x,B_{\varphi_1}(x,z,q))$ 
     and $u\in N(z')$ 
     such that $\varphi_1(xz')=\beta'$,  $\varphi_1(z'u)=\beta$, and  $d(u)\ge q$.  Since
      $B_w(\varphi)=B_w(\varphi_1)$,  $B_w(\varphi)\subseteq B_{\varphi}(x,z,q)$, 
     $B_w(\varphi)\subseteq B_{\varphi_1}(x,z,q)$, and  $\varphi_1(xz')=\varphi(xz')$,  we see that $z'\in N(x,B_{\varphi_1}(x,z,q))$ with $\varphi_1(xz')=\beta'$ implies that 
     $z'\in  N(x,B_{\varphi}(x,z,q))$. Furthermore, since 
      $\varphi_1(xz')=\varphi(xz')$ and $\varphi_1(z'u)=\varphi(z'u)$, we see that the assertion in Lemma~\ref{bbcolor} holds. 
\end{proof}	
	
	Combining Lemma~\ref{LEM:degree3} and Lemma~\ref{bbcolor}, we obtain the 
	following result. 
\begin{COR}\label{bbcolor2}
		Let $G$ be a $\Delta$-critical graph, $\ve\in (0,1)$  and $q=(1-\ve)\Delta$, and let
	$xy\in E(G)$ with $d(x)<\ve\Delta$ and $\varphi\in \CC^\Delta(G-xy)$.
	If $(\bigcup_{z\in N(x, M_\varphi)}\pb(z))\cap B_\varphi(x,y,q)=B_\varphi(x,y,q)$ and $|B_\varphi(x,y,q)|=2$, then 
	 for at least one $z'\in N(x,B_\varphi)$ and for any $u\in N(z')-\{x\}$,  $d(u)\ge q$.
\end{COR}	

\section{Proof of Theorem~\ref{THM-main}}


\begin{proof}[Proof of Theorem~\ref{THM-main}]
	Let $G$ be a $\Delta$-critical graph of order $n$ with minimum degree at least $d$ and maximum degree $\Delta\ge (d+2)^{5d+10}$.  In fact, the proof only needs $\Delta \ge D_0$
	with $D_0$ defined in Equation~\eqref{D0}, where  $D_0\le (d+2)^{5d+10}$ by calculations with $\ve$ defined in~\eqref{def-of-q}, 
	since by Equation \eqref{c0} and Equation \eqref{def-of-q} , $c_0\le d/2$.
	By the definition of $D_0$ in \eqref{D0}, 
	\begin{eqnarray*}
	D_0&=&	\max\{f(\ve), \frac{3c_0+2}{\lambda^2}, \frac{N+1}{\ve^3}\}\\
	&= & \frac{N+1}{\ve^3} \le (\frac{d}{2}+2)(\frac{(d+2)^3}{4}+1)^{1.5d+2} (\frac{d+2}{2})^3\\
	&<& (d+2)^4(d+2)^{4.5d+6}<(d+2)^{5d+10}.
	\end{eqnarray*}

	Let $X$ be a largest independent set and let $Y=V(G)-X$. Note that $Y$ is not an independent set in $G$. For otherwise $G$ is a bipartite graph which is not $\Delta$-critical. Define 
	\begin{equation}
	\omega= \left.
	\begin{cases} \label{omega}
	2, & \text{if $d=3,4$; }  \\
	\\ 
	\sqrt[3]{(d-1)d}, & \text{if $d\ge 19$,}
	\end{cases}
	\right.
	\end{equation}
and define 	
	\begin{eqnarray}\label{def-of-q}
q=\frac{d+2-\omega}{d+2}\Delta, & \lambda =\frac{\omega ^3}{2(d+2)^3}, &\quad\text{and}\quad \ve=\frac{\omega}{d+2}.
	\end{eqnarray}
Denote by   
$$
\begin{array}{ll}
X^{++}=\{x\in X: d(x)=\Delta\}, &  \\
X^+\,\,\,=\{x\in X: q\le d(x)<\Delta\},\\
X^{-}_1\,\,\,=\{x\in X: \ve \Delta \le d(x)<q\}, & \\
X^-_2\,\,\,=\{x\in X: d\le d(x)<\ve\Delta\},& \quad  \,\text{if $d\ge 19$,} \\
X^-_2\,\,\,=\{x\in X: 3d-3\le d(x)<\ve\Delta\}, & \quad  \, \text{if $d=3, 4$,}\\
X_3^-\,\,\,=\{x\in X: d\le d(x)<3(d-1)\},& \quad  \, \text{if $d=3, 4$,} \\
X^-\,\,\,=\{x\in X: d\le d(x)<q\}. &
 \end{array}
$$
Since $G$ has minimum degree at least $d$, by the definitions above, we see that $X=X^{++}\cup X^+\cup X^-$, 
 for $d\ge 19$, $X^-=X_1^-\cup X_2^-$,  and for $d=3,4$, $X^-=X_1^-\cup X_2^-\cup X_3^-$.

For each positive integer $k$,  define 
\begin{eqnarray}\label{def-of-g}
g_1(k)=\frac{(d+2)(\Delta -k)}{k}, & \quad \text{and}\quad g_2(k)=\frac{\omega \Delta}{k-1}.	
\end{eqnarray}
Clearly, $g_1(k)$ and $g_2(k)$ are both decreasing functions of $k$, and  
\begin{eqnarray}\label{g1q}
g_1(q)=\frac{(d+2)(\Delta -q)}{q}=\frac{\omega (d+2)}{ d+2-\omega}.	
\end{eqnarray}

\begin{CLA}\label{Claim1}
Let  $x\in X^+$ and $k=d(x)$.  Then $g_1(k)\le g_2(k)$. 
\end{CLA}
\pf Let $g(k)=g_2(k)-g_1(k)=\frac{\omega \Delta}{k-1}-\frac{(d+2)(\Delta -k)}{k}$. 
By calculation, the first order derivative of $g(k)$ is $g'(k)=\frac{(k-1)^2(d+2)\Delta-k^2\omega \Delta}{(k-1)^2k^2}$. Since $\omega\le \frac{(k-1)^2(d+2)}{k^2}$, $g'(k)\ge 0$. 
Thus, since $k\ge q$ if $x\in X^+$, 
\begin{eqnarray*}
	g(k)\ge g(q) & =& g_2(q)-g_1(q),\\
	               &=& \frac{\omega \Delta}{q-1}-\frac{\omega \Delta}{q}>0.
\end{eqnarray*}
Hence, $g_1(k)\le g_2(k)$.
\qqed 

Define three charge functions on $V(G)$ as follows. 
$$
\begin{array}{llll}
	M_0(x)=0, & M_1(x)=(d+2)d(x), & M_2(x)=(d+2)\Delta, &\text{if $x\in X$}\\
M_0(y)=(d+2+\omega)\Delta, & M_1(y)=\omega \Delta, & M_2(y)=0, &\text{if $y\in Y$}.
\end{array}
$$

Starting with the distribution $M_0$, we redistribute charge according to the following {\bf Discharging Rule}:
\begin{flushright}     
\begin{enumerate}
	\item []{\bf Step 0}: Each vertex $y\in Y$ gives charge $d+2$ to each vertex $x\in N(y)\cap X$. 
	Denote the resulting charge by $M_0^*$. 
	\item []{\bf Step 1}: Stating with $M_1$, each vertex $y\in Y$ gives charge $g_1(d(x))$ to each 
	$x\in N(y)\cap X^+$.  Denote the resulting charge by $M_1^*$.  
	
	\item []{\bf Step 2}: Stating with $M_1^*$, for each vertex $y\in Y$, if $M_1^*(y)>0$, $y$ distributes its remaining charge equally among all vertices (if any) in 
	$x\in N(y)\cap X^-$.  Denote the resulting charge by $M_2^*$. 
\end{enumerate}
\end{flushright} 

\begin{CLA}\label{Claim2}
	For each $v\in V(G)$, if $M_2^*(v)\ge M_2(v)$, then $\alpha(G)<\frac{d+2+\omega}{2d+4+\omega}n$.  Consequently, Theorem~\ref{THM-main} holds. 
\end{CLA}
\pf By Step 0 of Discharging Rule, 
\begin{eqnarray*}
	M_0^*(x) \quad =& \sum\limits_{y\in N(x)}(d+2)=(d+2)d(x)=M_1(x), & \text{for each $x\in X$;}\\
		M_0^*(y) \quad =& M_0(y)-\sum\limits_{x\in N(y)\cap X}(d+2) & \\
		\ge &(d+2+\omega)\Delta-(d+2)\Delta=\omega\Delta=M_1(y), &\text{for each $y\in Y$.}
\end{eqnarray*}
Since $G$ is $\Delta$-critical and so it is not bipartite, there exists $y\in Y$ so that 
$|N(y)\cap X|<\Delta$ and thus $M_0^*(y)>M_1(y)$. Hence, 
$$
\sum\limits_{v\in V(G)}M_1(v)<\sum\limits_{v\in V(G)}M_0^*(v)=\sum\limits_{v\in V(G)}M_0(v)=(d+2+\omega)\Delta|Y|. 
$$
For each $v\in V(G)$, if $M_2^*(v)\ge M_2(v)$, since $M_2^*$ is obtained based on $M_1$ by Steps 1 and 2 of  Discharging Rule, then we have that 
$$
(d+2)\Delta|X|=\sum\limits_{v\in V(G)}M_2(v)\le \sum\limits_{v\in V(G)}M_2^*(v)=\sum\limits_{v\in V(G)}M_1(v)<(d+2+\omega)\Delta|Y|. 
$$
The above inequality implies that 
$\alpha(G)=|X|<\frac{d+2+\omega}{2d+4+\omega}n$. Plugging in the values of $\omega$ in $\frac{d+2+\omega}{2d+4+\omega}n$ gives the desired bounds on $\alpha(G)$ in Theorem~\ref{THM-main}.
\qqed 

By Claim~\ref{Claim2}, we only need to show that for each $v\in V(G)$, $M_2^*(v)\ge M_2(v)$ holds. We show this by considering  different cases according to which set $v$ belongs to. 

\begin{CLA}\label{Claim3}
	For each $y\in Y$,  $M_2^*(y)\ge M_2(y)=0$. 
\end{CLA}

\pf Let $y\in Y$, and let $k_0=\min\{d(x)\,:\, x\in N(y)\cap X^+\}$. 
By Lemma~\ref{vizing adjacency lemma}, $y$ is adjacent to at least $\Delta -k_0+1$ neighbors of degree $\Delta$. 
Thus $y$ is adjacent to at most $d(y)-(\Delta -k_0+1)\le k_0-1$ neighbors in $X^+\cup X^-$. 

By Step 2 in Discharging Rule, to show $M_2^*(y)\ge 0=M_2(y)$, it suffices to show that 
$M_1^*(y)\ge 0$. 
By Step 1 in Discharging Rule,  we have that 
$$
M_1^*(y)=M_1(y)-\sum\limits_{x\in N(y)\cap X^+}g_1(d(x)).
$$
By Claim~\ref{Claim1}, for $x\in X^+$, $g_1(d(x))\le g_2(d(x))$. 
Since $g_2(k)$ is decreasing  and $k_0$ is the minimum value among the degrees of $x$ in $N(y)\cap X^+$,  $g_2(d(x))\le \frac{\omega \Delta}{k_0-1}$. Combining the arguments above,  
we get that
\begin{eqnarray*}
	M_1^*(y)& =& M_1(y)-\sum\limits_{x\in N(y)\cap X^+}g_1(d(x))\ge M_1(y)-\sum\limits_{x\in N(y)\cap X^+}g_2(d(x))\\
& \ge & M_1(y)-|N(y)\cap X^+|\frac{\omega \Delta}{k_0-1} \ge M_1(y)-(k_0-1)\frac{\omega \Delta}{k_0-1} \\
	&= & \omega \Delta -\omega\Delta =0.
\end{eqnarray*}
\qqed 

\begin{CLA}\label{Claim4}
	For each $x\in X^{++}\cup X^+$,  $M_2^*(x)\ge M_2(x)=(d+2)\Delta$. 
\end{CLA}
\pf For each $x\in X^{++}$, by Step 0,  we have that 
$$
M_0^*(x)=\sum\limits_{y\in N(x)\cap Y}(d+2)=(d+2)\Delta,
$$
where we get  $|N(x)\cap Y|=\Delta$ since $X$ is an independent set in $G$.
The charge of $x\in X^{++}$ keeps unchanged in Steps 1 and 2, thus 
$M_2^*(x)=M_0^*(x)=(d+2)\Delta$. 

For each $x\in X^+$, by Discharging Rule, 
$$
M_2^*(x)=M_1^*(x)=M_1(x)+\sum\limits_{y\in N(x)} g_1(d(x))=(d+2)d(x)+\sum\limits_{y\in N(x)}\frac{(d+2)(\Delta -d(x))}{d(x)}=(d+2)\Delta. 
$$
\qqed 

The next claim will be used for showing that for each $x\in X^-$, $M_2^*(x)\ge M_2(x)=(d+2)\Delta$. 

\begin{CLA}\label{Claim5}
Let $\ell$ be a nonnegative integer and $y\in Y$ be a neighbor of $x\in X^-$, and  
$k=d(x)$.  If $\sigma_q(x,y)\ge \Delta -k+1+\ell$, then $y$ gives $x$ 
at least 
$$
h(k,\ell)=\frac{1}{k-\ell-1}(\omega \Delta-\ell g_1(q))=\frac{1}{k-\ell-1}(\omega \Delta-\ell \frac{(d+2)\omega}{d+2-\omega})
$$
in Step 2. 
\end{CLA}

\pf Let $L^{++}$ be a set of $\Delta -k+1$ neighbors of $y$ with degree $\Delta$, and let $L^+$ be a set, disjoint from $L^{++}$, of $\ell$ neighbors of $y$ with degree at least $q$, which exists since $\sigma_q(x,y)\ge \Delta -k+1+\ell$. Then in Steps 1 and 2, $y$ gives nothing to vertices in $L^{++}$,
and in Step 1, for each vertex $x'\in N(y)\cap L^+$, $y$ gives $g_1(d(x'))\le g_1(q)$ to $x'$. In Step 2, $y$'s remaining charge of at least $\omega \Delta -\ell g_1(q)$ is divided among $y$'s remaining $d(y)-(\Delta -k+1+\ell)\le k-\ell -1$ neighbors. Denote the set of these remaining neighbors of $y$ by $L$. 
For $x$, being in $X^-$, receives charge of at least 
$\frac{\omega \Delta-\ell g_1(q)}{k-\ell -1}$. This is because  of the following arguments.
(a) For each $x'\in L\cap X^+$, $g_1(d(x'))\le g_1(q)\le g_2(q)\le g_2(k)=\frac{\omega \Delta}{k-1}$ holds, and thus
\begin{eqnarray*}
	\frac{\omega \Delta -\ell g_1(q)}{k-\ell-1}&\ge &	\frac{\omega \Delta -\ell g_2(q)}{k-\ell-1}\ge \frac{\omega \Delta -\ell g_2(k)}{k-\ell-1}\\
	& = & \frac{\omega \Delta -\frac{\ell \omega \Delta}{k-1}}{k-\ell-1}=g_2(k).
\end{eqnarray*}
(b) Thus by (a),  $g_1(d(x'))\le g_1(q)\le g_2(q)\le g_2(k)\le \frac{\omega \Delta -\ell g_1(q)}{k-\ell-1}$, 
and therefore the charge that $y$ gives to $x$ is 
\begin{eqnarray*}
	\frac{\omega \Delta -\ell g_1(q)-\sum\limits_{x'\in L\cap X^+}g_1(d(x'))}{|L|-|L\cap X^+|}&\ge &	\frac{\omega \Delta -\ell g_1(q)-|L\cap X^+|\frac{\omega \Delta -\ell g_1(q)}{k-\ell-1}}{|L|-|L\cap X^+|}\\
	& \ge  & \frac{\omega \Delta -\ell g_1(q)-|L\cap X^+|\frac{\omega \Delta -\ell g_1(q)}{k-\ell-1}}{k-\ell-1-|L\cap X^+|}= \frac{\omega \Delta -\ell g_1(q)}{k-\ell-1}.
\end{eqnarray*}
\qqed

\begin{CLA}\label{Claim6}
	For each $x\in X_1^{-}$,  $M_2^*(x)\ge M_2(x)=(d+2)\Delta$. 
\end{CLA}

\pf Let $x\in X_1^-$ and let $k=d(x)$. Note that by the definition of $X_1^-$, $\ve \Delta\le k<q=(1-\ve)\Delta$. Let 
$$
p=\min\limits_{y'\in N(x)}\{\sigma_q(x,y')-(\Delta -k+1)\}. 
$$
Assume that $y\in N(x)$ achieves  $\sigma_q(x,y)=(\Delta -k+1)+p$. 
Since $d(x)<q$,  we have $d(u)\ge q$ for any $u\in N(y)$ such that 
$\varphi(yu)\in \pbar(x)$. Thus, because $|\pbar(x)|=\Delta -k+1$, by the definition of $B_\varphi(x,y,q)$, 
we have that $|B_\varphi(x,y,q)|=p$. 
Since $M_\varphi(x,y,q)\cap \varphi(x)=A_\varphi(x,y,q)\cup \pbar(y)$,  we get that 
$|M_\varphi(x,y,q)\cap \varphi(x)|\ge \Delta-\sigma_q(x,y)= k-p-1.$
By Lemma~\ref{LEM:largedegree}, $x$ has at least $k-1-p-N$ neighbors $z\in N(x, M_\varphi)$ such that 
$\sigma_q(x,z)\ge \Delta-1-(\lambda \Delta-1) =\Delta-\lambda \Delta  +(\Delta -k+1)-(\Delta -k+1)= (\Delta -k+1)+k-\lambda \Delta -1$,
where  $c_0=\lceil \frac{1-\ve}{\ve}\rceil$,  $\lambda =\frac{\omega ^3}{2(d+2)^3}$, and
$N=(c_0+2)(\frac{1}{\lambda}+1)^{3c_0+2}$.  
The remaining $p+1+N$ neighbors $z'$ of $x$ satisfies 
$\sigma_q(x,z')\ge \Delta -k+1+p$.  By Claim~\ref{Claim5}, in Step 2, $x$ receives charge of 
at least 
$$
M(k,p)=(k-p-1-N)h(k,k-\lambda \Delta -1)+(p+1+N)h(k,p).
$$
By Step 0, we have that $M_1(x)=(d+2)k$, so we only need to show that 
$M(k,p)\ge (d+2)(\Delta -k)$.  Recall that $h(k,\ell)=\frac{1}{k-\ell-1}(\omega \Delta-\ell g_1(q))$.
Then 
\begin{eqnarray*}
	M(k,p)&=& (k-p-1-N)h(k,k-\lambda \Delta -1)+(p+1+N)h(k,p)\\
	&= & \frac{k-p-1-N}{\lambda \Delta +1}\left(\omega \Delta -kg_1(q)+(\lambda \Delta +1)g_1(q)\right)\\
	&&+\frac{p+1+N}{k-p-1}(\omega \Delta -pg_1(q))\nonumber \\
	&=& \frac{k-1-p-N}{\lambda \Delta +1}(\omega \Delta -kg_1(q)+(\lambda \Delta +1)g_1(q))\\
	&&+\frac{p+1+N}{k-p-1}(\omega \Delta -kg_1(q)) +(p+1+N)g_1(q)\nonumber \\
	&=& \frac{k-p-1}{\lambda \Delta +1}(\omega \Delta -kg_1(q))-\frac{N}{\lambda \Delta +1}(\omega \Delta -kg_1(q))\\
	&&+\frac{k+N}{k-p-1}(\omega \Delta -kg_1(q))-(\omega \Delta -kg_1(q))+kg_1(q)\\
	&>&2\sqrt{\frac{k}{\lambda \Delta +1}}(\omega \Delta -kg_1(q))-\frac{N}{\lambda \Delta +1}(\omega \Delta -kg_1(q))-(\omega \Delta -kg_1(q))+kg_1(q)\\
		&=&\left( 2\sqrt{\frac{k}{\lambda \Delta +1}}-\frac{N}{\lambda \Delta +1}-1\right)(\omega \Delta -kg_1(q))+kg_1(q).
	\end{eqnarray*} 
Let 
\begin{eqnarray*}\label{f(k)}
	f(k)&=& \left( 2\sqrt{\frac{k}{\lambda \Delta +1}}-\frac{N}{\lambda \Delta +1}-1\right)(\omega \Delta -kg_1(q))+kg_1(q) -(d+2)(\Delta -k)\\
	&=&  \left( 2\sqrt{\frac{k}{\lambda \Delta +1}}-\frac{N}{\lambda \Delta +1}-1\right)(\omega \Delta -kg_1(q))-\frac{d+2}{\omega}(\omega \Delta -kg_1(q))\\
	&=& \left( 2\sqrt{\frac{k}{\lambda \Delta +1}}-\frac{N}{\lambda \Delta +1}-1-\frac{d+2}{\omega}\right)(\omega \Delta -kg_1(q)),
\end{eqnarray*} 
where $d+2+g_1(q)=(d+2)(1+\frac{\omega}{d+2-\omega})=\frac{d+2}{\omega}g_1(q)$.
Since $\omega \Delta -kg_1(q)\ge 0$ always holds, to show that $f(k)\ge 0$, it suffices
to show that $f(\ve \Delta)\ge 0$.   
Recall that $\ve=\frac{\omega}{d+2}$, $\lambda=\frac{\omega^3}{2(d+2)^3}$, 
and $\Delta \ge \frac{N+1}{\ve^3}$. Thus we get that 
\begin{eqnarray*}\label{f(k)}
	f(\ve \Delta)
	&=& \left( 2\sqrt{\frac{\ve \Delta}{\lambda \Delta +1}}-\frac{N}{\lambda \Delta +1}-1-\frac{d+2}{\omega}\right)(\omega \Delta -\omega \Delta g_1(q))\\
	&>& \left(2.5\frac{d+2}{\omega}-\frac{N}{\frac{\omega^3}{2(d+2)^3}\frac{N+1}{\frac{\omega^3}{(d+2)^3}}}-1-\frac{d+2}{\omega}\right)(\omega \Delta -\omega \Delta g_1(q))\\
	&>& (1.5\frac{d+2}{\omega}-3)(\omega \Delta -\omega \Delta g_1(q))>0. 
\end{eqnarray*} 

\qqed 

\begin{CLA}\label{Claim7}
	For each $x\in X_2^{-}$,  $M_2^*(x)\ge M_2(x)=(d+2)\Delta$. 
\end{CLA}

\pf Let $x\in X_2^-$ and let $k=d(x)$. Note that by the definition of $X_2^-$, $6\le k<\ve\Delta$ if $d=3$, and $9\le k<\ve\Delta$ if $d=4$, and $d\le k<\ve\Delta$ if $d\ge 19$. Let 
$$
p=\min\limits_{y'\in N(x)}\{\sigma_q(x,y')-(\Delta -k+1)\}. 
$$
Assume that $y\in N(x)$ achieves  $\sigma_q(x,y)=(\Delta -k+1)+p$. Then 
$|M_\varphi(x,y,q)\cap \varphi(x)|\ge d(y)-(\Delta-k+1)-p\ge k-p-1.$
By Lemma~\ref{LEM:smalldegree}, $x$ has at least $k-1-p$ neighbors $z\in N(x, M_\varphi)$  with $\varphi(xz)=\alpha$ such that
$\pb(z)-\{\alpha\}\subseteq B_\varphi(q)$, and for any $z_1,z_2\in N(x, M_\varphi)$, $\pb(z_1)\cap \pb(z_2)=\emptyset$. 
For each $z_i\in N(x, M_\varphi)$, let 
$$
b_i=|\pb(z_i)-\{\varphi(xz_i)\}|.
$$
Then by the definition of $\pb(z_i)$, we have that $\sigma_q(x,z_i)+b_i=\Delta-1$. Thus 
$\sigma_q(x,z_i)\ge \Delta -b_i-1$ for each $z_i\in N(x, M_\varphi)$. 
By Lemma~\ref{LEM:smalldegree}
$\sum\limits_{z_i\in N(x, M_\varphi)}b_i \le |B_\varphi(q)|= p$.
Let
$$
\ell_i=\sigma_q(x,z_i)-(\Delta -k+1)=k-b_i-2 \quad \text{and}\quad h_i(k,\ell_i)=\frac{1}{b_i+1}(\omega \Delta -(k-1)g_1(q))+g_1(q).
$$
The remaining $p+1$ neighbors $z'$ of $x$ satisfies 
$\sigma_q(x,z')\ge \Delta -k+1+p$.  By Claim~\ref{Claim5}, in Step 2, $x$ receives charge of 
\begin{eqnarray}\label{smallk}
M(k,p)&=&\sum\limits_{i=1}^{k-p-1}h_i(k,\ell_i)+(p+1)h(k,p) \\
& =& \sum\limits_{i=1}^{k-p-1}\frac{1}{b_i+1}(\omega \Delta -(k-1)g_1(q))+(k-p-1)g_1(q) + \frac{p+1}{k-p-1}(\omega \Delta -p g_1(q)).\nonumber 
\end{eqnarray}
We claim that $\sum\limits_{i=1}^{k-p-1}\frac{1}{b_i+1}\ge \frac{k-p-1}{1+\frac{p}{k-p-1}}=\frac{(k-p-1)^2}{k-1}$. This follows by applying Cauchy-Schwarz inequality as below by noticing that  $\sum\limits_{z_i\in N(x, M_\varphi)}b_i \le p$,
\begin{eqnarray*}
	\sum\limits_{i=1}^{k-p-1}\frac{1}{b_i+1}(k-1)&=&\sum\limits_{i=1}^{k-p-1}\frac{1}{b_i+1}(p+k-p-1)\ge \sum\limits_{i=1}^{k-p-1}\frac{1}{b_i+1}\sum\limits_{i=1}^{k-p-1}(b_i+1)\\
		&\ge&\left(\sum\limits_{i=1}^{k-p-1}\sqrt{\frac{1}{b_i+1}(b_i+1)}\right)^2=(k-p-1)^2.
\end{eqnarray*} 
Thus, 
\begin{eqnarray*}
	M(k,p)&\ge & \frac{(k-p-1)^2}{k-1}(\omega \Delta -(k-1)g_1(q))+(k-p-1)g_1(q) + \frac{p+1}{k-p-1}(\omega \Delta -p g_1(q))\\
	&> & \frac{(k-p-1)^2}{k-1}(\omega \Delta -kg_1(q))+(k-p-1)g_1(q) \\
	&&+ \frac{p+1}{k-p-1}(\omega \Delta -k g_1(q))+(p+1)g_1(q)\\
		&= & \frac{(k-p-1)^2}{k-1}(\omega \Delta -kg_1(q))+(k-p-1)g_1(q) \\
	&&+ \frac{k}{k-p-1}(\omega \Delta -k g_1(q))+(p+1)g_1(q)-(\omega \Delta -k g_1(q))\\
	&=& \frac{(k-p-1)^2}{k-1}(\omega \Delta -kg_1(q))+kg_1(q)-(\omega \Delta -k g_1(q))\\
	&&+\frac{k}{2(k-p-1)}(\omega \Delta -k g_1(q))+\frac{k}{2(k-p-1)}(\omega \Delta -k g_1(q)) \\
	&\ge & 3 \sqrt[3]{\frac{k^2}{4(k-1)}}(\omega \Delta -kg_1(q))-(\omega \Delta -k g_1(q))+kg_1(q)
\end{eqnarray*} 
Let 
\begin{eqnarray*}\label{f(k)2}
	f(k)&=& \left(3\sqrt[3]{\frac{k^2}{4(k-1)}}-1\right)(\omega \Delta -kg_1(q))+kg_1(q)-(d+2)(\Delta -k)\\
	&=& \left(3\sqrt[3]{\frac{k^2}{4(k-1)}}-1\right)(\omega \Delta -kg_1(q))-\frac{d+2}{\omega}(\omega \Delta -kg_1(q))\\
	&=&\left(3\sqrt[3]{\frac{k^2}{4(k-1)}}-1-\frac{d+2}{\omega}\right)(\omega \Delta -kg_1(q)).
\end{eqnarray*}
Let $f_1(k)=\left(3\sqrt[3]{\frac{k^2}{4(k-1)}}-1-\frac{d+2}{\omega}\right)$. 
Since  $\omega \Delta -kg_1(q)>0$ always holds,  we check that $f_1(3d-3)\ge 0$ when $d=3,4$ and $f_1(d)\ge 0$ when $d\ge 19$. 
Since $\Delta >(d+2)^5$, $\omega =2$ when $d=3,4$, and $\omega =\sqrt[3]{d(d-1)}<\frac{d}{2}$ when $d\ge 19$, we get that 
\begin{eqnarray*}\label{f(k)2}
	f_1(6)&>& 3.649-1-2.5>0, \quad \text{when $d=3$},\\
	f_1(9)&>& 4.08-1-3>0, \quad \text{when $d=4$},\\
	f_1(d)&=& \frac{3d\sqrt[3]{\frac{1}{4}}}{\omega}-1-\frac{d+2}{\omega}>\frac{1.88d}{\omega}-\frac{d+2+\omega}{\omega}>0,  \quad \text{when $d\ge 19$}.
\end{eqnarray*}
\qqed 

\begin{CLA}\label{Claim8}
	For each $x\in X_3^{-}$,  $M_2^*(x)\ge M_2(x)=(d+2)\Delta$. 
\end{CLA}
\pf Let $x\in X_3^-$ and let $k=d(x)$. Note that by the definition of $X_3^-$, $3\le k<3(d-1)$ if $d=3$ and $4\le k<3(d-1)$ if $d=4$. Let 
$$
p=\min\limits_{y'\in N(x)}\{\sigma_q(x,y')-(\Delta -k+1)\}. 
$$
Assume that $y\in N(x)$ achieves  $\sigma_q(x,y)=(\Delta -k+1)+p$. 


When $d=3$, for $k=3,4,5$, and when $d=4$, for  $k=4, 5, 6,7,8$ 
$$\text {except}  \quad  (k,p)\in \{(4,1), (5,2), (6,3), (7,3)\},
$$ we use the same charge function $M(k,p)$ in Equation~\eqref{smallk} as in Claim~\ref{Claim7}. 
Note that for any $z\in N(x)$, 
$z$ is adjacent to at most $\Delta -(\Delta -k+1)-1-p=k-2-p$ vertices other than $x$ of degree less than $q$. Thus, when $p=k-2$, $\ell \ge k-2$, directly by Lemma~\ref{Claim5}:
\begin{eqnarray}\label{large}
M(k,k-2)&\ge &(p+2)(\omega \Delta -pg_1(q))=(k+2)\Delta +(k-2)(\Delta -kg_1(q))\\
&>&(d+2)(\Delta-k). \nonumber
\end{eqnarray}
Also, we obtain the following values by plugging in the value for $p$ in Equation ~\eqref{smallk}:
\begin{eqnarray}
M(k,0)&=&(k-1)(\omega \Delta -(k-2)g_1(q))+\frac{\omega \Delta }{k-1}>(d+2)(\Delta -k). \label{smallp}\\
M(k,1)&\ge& (k-3+\frac{1}{2})(\omega \Delta -(k-1)g_1(q))+\frac{2 }{k-2}(\omega \Delta-g_1(q)) \label{smallp2} \\
&>&(d+2)(\Delta -k). \quad \quad(\text{ $5\le k\le 8$ or $d=3$ and $k=4$}) \nonumber\\
M(k,2)&\ge& (k-5+\frac{1}{2}+\frac{1}{2})(\omega \Delta -(k-1)g_1(q))+\frac{3 }{k-3}(\omega \Delta-2g_1(q)) \label{smallp3} \\
&& (\mbox{the minimum is attained when two of the $b_i$'s are one})\\
&\ge & (d+2)(\Delta -k).  \quad \quad(\text{ $6\le k\le 8$ or $d=3$ and $k=5$})   \nonumber\\
M(k,k-3)&\ge& \frac{4}{k-1}(\omega \Delta -(k-1)g_1(q))+\frac{k-2 }{2}(\omega \Delta-(k-3)g_1(q)) \label{smallp4}\\
&& (\mbox{taking $b_1=b_2=\frac{k-3}{2}$})\\
&>&  6(\Delta -k).  \quad \quad(\text{ $7\le k\le 8$})   \nonumber
\end{eqnarray}
From Equations~\eqref{large} to \eqref{smallp4}, we only need to verify that $M(8,3)\ge 6(\Delta-8)$ and 
$M(8,4)\ge 6(\Delta-8)$ when $d=4$. Indeed,  
\begin{eqnarray*}
	M(8,3)&\ge & 7\Delta -49.5>6(\Delta -8),\\
	M(8,4)&\ge & \frac{20\Delta}{3} -46>6(\Delta -8).
\end{eqnarray*}
We are now left to check that when $d=4$ and $x\in X_3^-$ with $d(x)=k$, and 
 when $ (k,p)\in \{(4,1), (5,2), (6,3), (7,3)\}$,
 $M_2^*(x)\ge 6(\Delta -k)$. 
The charge function $M(k,p)$ in Equation~\eqref{smallk} does not 
give the desired bounds in general, we seek a different approach for the proof. 

For $k=4$ and $p=1$, if there exists $z\in N(x, M_\varphi)$ such that $\pb(z)\cap B_\varphi(x,y,q)\ne \emptyset$, then by Corollary~\ref*{p=1}, we have that 
$x$ has two neighbors such that each of them is adjacent to at most one vertex other than $x$
of degree less than $q$ and the other two neighbors of $x$ each is only adjacent to exactly one vertex of degree less than $q$ which is $x$.
Thus, similar as in Claim~\ref{Claim7}, we have that 
\begin{eqnarray*}
M(4,1)& =& 2(\omega \Delta -(4-2)g_1(q))+2(\frac{\omega \Delta -pg_1(q)}{2})  =6\Delta -18>6(\Delta -6).
\end{eqnarray*}
If for any $z\in N(x, M_\varphi)$, it holds that $\pb(z)\cap B_\varphi(x,y,q)= \emptyset$, then we  have the same charge function as above. 

We then consider $d(x)=k=5$ and $p=2$. We have that $|N(x,M_\varphi)|=2$ and by Lemma~\ref{LEM:smalldegree},
 for $z_1, z_2\in N(x,M_\varphi)$, $\pb(z_1)\cap \pb(z_2)=\emptyset$. 
 If $|(\pb(z_1)\cup \pb(z_2))\cap B_\varphi(x,y,q)|\le 1$, then we know that 
 among the 5 neighbors of $x$, one neighbor of $x$ is adjacent to exactly one vertex of degree less than $q$ which is $x$, and for the other 4 neighbors of $x$, each of them is adjacent to exactly one vertex other than $x$ of degree less than $q$. Hence, by Claim~\ref{Claim5}, we get that 
  \begin{eqnarray*}
  	M(5,2)& =& (\omega \Delta -(5-2)g_1(q))+4(\frac{\omega \Delta -pg_1(q)}{2}) =6\Delta -21>6(\Delta -6).
  \end{eqnarray*}
 Thus, we assume that  $(\pb(z_1)\cup \pb(z_2))\cap B_\varphi(x,y,q)=B_\varphi(x,y,q)$. 
Now by Corollary~\ref{bbcolor2}, there exists $z\in N(x, B_\varphi)$ such that $z$ is adjacent to exactly one vertex of degree less than $q$ which is $x$. 

We next consider $(k,p)=(6,3)$ and $(k,p)=(7,3)$. We have that $|N(x,M_\varphi)|=k-p-1$ and by Lemma~\ref{LEM:smalldegree},
for $z_1, z_2\in N(x,M_\varphi)$, $\pb(z_1)\cap \pb(z_2)=\emptyset$. 
If $|(\pb(z_1)\cup \pb(z_2))\cap B_\varphi(x,y,q)|\le 2$,  by Equation~\eqref{smallk}, we get that 
\begin{eqnarray*}
	M(6,3)& \ge & (\omega \Delta -5g_1(q)) + 2g_1(q)+2(\omega \Delta -3g_1(q))=6\Delta -27>6(\Delta -6);\\
		M(7,3)& \ge & 2(\omega \Delta -6g_1(q)) + 3g_1(q)+\frac{4}{3}(\omega \Delta -3g_1(q))=\frac{20}{3}\Delta -39>6(\Delta -7). 
\end{eqnarray*}
Thus, we assume that  $(\pb(z_1)\cup \pb(z_2))\cap B_\varphi(x,y,q)=B_\varphi(x,y,q)$. 
Now by Lemma~\ref{bbcolor}, there exists $z\in N(x, B_\varphi)$ such that $z$ is adjacent to at most one vertex of degree less than $q$ other than $x$. 
By Claim~\ref{Claim5}, we get the same charge function as above.  
%
\qqed 

The proof of Theorem~\ref{THM-main} is now complete. 
\end{proof}
\bibliographystyle{plain}
\bibliography{SSL-BIB}

\end{document}